\documentclass{article}
\usepackage{amsmath,amsthm,amssymb,amscd}

\newcommand{\ga}{\alpha}
\newcommand{\gb}{\beta}

\newcommand{\gd}{\delta}
\newcommand{\gw}{\omega}

\newcommand{\eps}{\varepsilon}

\newcommand{\cantor}{2^\gw}

\newcommand{\supp}{\mathrm{supp}}

\newcommand{\dom}{\mathrm{dom}}

\newcommand{\acts}{\curvearrowright}
\newcommand{\hvd}{\mathrm{HOD}}
\newcommand{\vd}{\mathrm{OD}}

\newcommand{\forkindep}[1][]{%
  \mathrel{
    \mathop{
      \vcenter{
        \hbox{\oalign{\noalign{\kern-.3ex}\hfil$\vert$\hfil\cr
              \noalign{\kern-.7ex}
              $\smile$\cr\noalign{\kern-.3ex}}}
      }
    }\displaylimits_{#1}
  }
}

\newtheorem{theorem}{Theorem}[section]

\newtheorem{claim}[theorem]{Claim}
\newtheorem{corollary}[theorem]{Corollary}
\newtheorem{fact}[theorem]{Fact}
\newtheorem{proposition}[theorem]{Proposition}

\theoremstyle{definition}
\newtheorem{definition}[theorem]{Definition}
\newtheorem{example}[theorem]{Example}

\title{Independence relations in the Solovay model II\footnote{2020 AMS subject classification 03E25, 22F05.}}

\author{
Jind{\v r}ich Zapletal\\
University of Florida\\
zapletal@ufl.edu}

\begin{document}
\maketitle

\begin{abstract}
I provide a novel axiomatization of the Solovay model and a purely geometric treatment of the theory of balanced forcing.
\end{abstract}

\section{Introduction}

\noindent In the previous instalment \cite{z:reloadedA}, I provided a ``geometric'' axiomatization of the Solovay model using an independence relation $\forkindep$ reminiscent of stable or simple theories in model theory. Its purpose lies in the ability to analyze the theory of the model without constantly referring back to its construction. In the current paper, I develop the theory of balanced forcing under the geometric axiomatization, which allows a detailed study of many generic extensions of the Solovay model. This methodological change turns out to be a great simplification and generalization of the treatment in \cite{z:geometric} and its sequels such as 
\cite{z:vitali, z:triangles, z:distance, z:ngraphs}. The basic terminology of these papers is retained, even though the new definitions are more general and stated in the choiceless context. 

As for the architecture of the paper, in Section~\ref{balancedsection} I provide a new definition of balanced forcing, and prove basic theorems for it. Section~\ref{examplesection} provides basic forcing examples. As in the case of proper forcing in ZFC, the class of balanced forcings needs to be stratified in order to obtain flexibility instrumental for most independence results. The stratification described in this paper uses independence relations which are variations and weakenings of the master independence relation $\forkindep$ of the Solovay model. Section~\ref{independencesection} defines the concept of an independence relation and derives its main features. The subsequent sections go through a list of independence relations which are interesting in their own right and which also appear in slight disguise in preceding work. Each of these sections defines an independence relation $\perp$ and its associated class of $\perp$-balanced posets, provides examples of $\perp$-balanced forcings, provides methods to create $\perp$-independent tuples often connected with continuous actions of Polish groups, and finally includes a typical preservation theorem for the class of $\perp$-balanced forcings with an associated consistency result. The treatment should not be viewed as definitive, as many other preservation theorems can readily be derived.

The base theory of this paper is the geometric axiomatization of the Solovay model as exposed in \cite{z:reloadedA}. It includes ZF+DC plus definability, inaccessibility, and Baire property axioms and critically, the existence of a master independence relation $\forkindep$ with several properties. The axiomatization together with its basic consequences is recorded in the appendix Section~\ref{axiomatizationsection} for reference purposes. The appendix also explains the details of parlance regarding Cohen-generic points of Polish spaces.

The notation follows the set theoretic standard of \cite{jech:newset} and the parlance of the previous installment \cite{z:reloadedA}. In particular, if $x$ is a set of ordinals then $\vd_x$ denotes the class of all sets definable from $x$ and a finite tuple of ordinals and $\hvd_x$ denotes the transitive part of $\vd_x$.If $P$ is a poset, an open set $O\subset P$ is \emph{regular} if for every $p\in P$ either $p\in O$ or there is a condition $q\leq p$ below which there are no conditions in $O$. The set of regular open subsets of $O$ is closed under arbitrary intersections. 

The work in this paper was partly supported by NSF grant  DMS 2348371.

\section{Balanced forcing: basics}
\label{balancedsection}

This section provides a definition of the class of balanced forcings under the geometric axiomatization.

\begin{definition}
Let $\langle P, \leq\rangle$ be a poset.

\begin{enumerate}
\item Let $x$ be a set of ordinals such that $\langle P, \leq\rangle\in \vd_x$. A nonempty open set $O\subset P$ is \emph{balanced over $x$} if it is in $\vd_x$ and whenever $y_0, y_1$ are sets of ordinals and $p_0, p_1\in O$ are conditions such that $y_0\forkindep[x]y_1$ and $p_0\in\vd_{xy_0}$ and $p_1\in\vd_{xy_1}$ then $p_0, p_1$ have a common lower bound in $P$;
\item the poset $P$ is \emph{balanced} if there is a set $x$ of ordinals such that for every set $y$ of ordinals such that $x\in\hvd_y$ and every condition $p\in P\cap\hvd_y$ there is a set $O\subset P$ balanced over $y$ which consists of conditions stronger than $p$;
\item the poset $P$ is \emph{cofinally balanced} if for every set $x$ of ordinals and every condition $p\in P$ there is a set $y$ of ordinals such that $x, p\in\hvd_y$ and a set $O\subset P$ balanced over $y$ which consists of conditions stronger than $p$.
\end{enumerate}
\end{definition}

\noindent For brevity of exposition, this paper only presents posets which are balanced. Many natural cofinally balanced, not balanced posets are used in \cite{z:geometric}; existence or nonexistence of balanced sets over a set $x$ of ordinals may depend on the internal theory of $\hvd_x$. The following proposition and corollary record a list of basic properties of balanced sets.

\begin{proposition}
\label{extensionproposition1}
Let $\langle P, \leq\rangle$ be a poset and $x$ a set of ordinals such that $\langle P, \leq\rangle\in\vd_x$. Let $O\subset P$ be a nonempty open set in $\vd_x$. The following are equivalent:

\begin{enumerate}
\item $O$ is balanced over $x$;
\item whenever $y_0, y_1$ are sets of ordinals such that $y_0\forkindep[x]y_1$ holds, and $O'_0\in\vd_{xy_0}$ and $O'_1\in\vd_{xy_1}$ are nonempty open subsets of $O$, then $O'_0\cap O'_1\neq 0$.
\end{enumerate}
\end{proposition}

\begin{proof}
The implication (2)$\to$(1) is immediate by considering the open sets $O'_0=\{q\in P\colon q\leq p_0\}$ and $O'_1=\{q\in P\colon q\leq p_1\}$. For (1)$\to$(2), suppose that $O$ is balanced over $x$ and $y_0, y_1, O'_0, O'_1$ are as in (2). Use the extension property of $\forkindep$ to find a set $z_0$ of ordinals such that $z_0\forkindep[xy_0]y_1$ and there is an element $p_0\in O'_0\cap \vd_{xy_0z_0}$. By the transitivity property of $\forkindep$, $y_0z_0\forkindep[x]y_1$ holds. Use the extension property of $\forkindep$ to find a set $z_1$ of ordinals such that $z_1\forkindep[xy_1]y_0z_0$ and there is an element $p_1\in O'_1\cap \vd_{xy_1z_1}$. By the transitivity property of $\forkindep$, $y_0z_0\forkindep[x]y_1z_1$ holds. By the balance assumption on $O$, the conditions $p_0$ and $p_1$ have a common lower bound, which is the desired element of $O'_0\cap O'_1$.
\end{proof}

\begin{corollary}
\label{opencorollary}
Let $\langle P, \leq\rangle$ be a poset and $x$ a set of ordinals such that $\langle P, \leq\rangle\in\vd_x$. Let $O\subset P$ be a balanced set over $x$.

\begin{enumerate}
\item any nonempty open subset of $O$ in $\vd_x$ is dense in $O$ and balanced over $x$;
\item the set $\{p\in P\colon O$ is dense below $p\}$ is balanced over $x$;
\item every condition in $O$ decides every statement of the $P$-forcing language with parameters in $\vd_x$, and all conditions in $O$ decide a given statement in the same way.
\end{enumerate}

\noindent Any two regular open subsets of $P$ which are balanced over $x$ are either disjoint or equal.
\end{corollary}

\begin{proof}
For (1), suppose $O_0\subset O$ is a nonempty open set in $\vd_x$. Its balance is clear from the definitions. For the density, consider $O_1=\{p\in O\colon$ for every $q\leq p$, $q\notin O_0\}$. This is an open subset of $O$ in $\vd_x$; if it is nonempty then by Proposition~\ref{extensionproposition1}, the intersection $O_0\cap O_1$ must be nonempty, which is absurd. Thus, $O_1$ is empty and $O_0$ must be dense in $O$. For (2), note that the set $\{p\in P\colon O$ is dense below $p\}$ inherits the criterion in Proposition~\ref{extensionproposition1}(2) from $O$. For (3), if $\phi$ is a statement of the $P$-forcing language with parameters in $\vd_x$ then the sets $O_0=\{p\in O\colon p\Vdash\phi\}$ and $O_1=\{p\in O\colon p\Vdash\lnot\phi\}$ are open and in $\vd_x$. By (1), if they are both nonempty, then they must have a nonempty intersection, which is absurd. Thus, one of these sets is empty, and the other is then equal to $O$.

For the last sentence, if $O_0\subset P$ and $O_1\subset P$ are balanced over $x$ and regular open, then their intersection is either empty or dense in both $O_0$ and $O_1$ by item (1). If the latter option occurs, the regularity of $O_0$ and $O_1$ implies that $O_0\cap O_1=O_0=O_1$.
\end{proof}

\noindent Corollary~\ref{opencorollary} (2) shows that every balanced set is a subset of a regular open balanced set; the last sentence then opens the possibility of classifying regular open sets which are balanced over $x$. This is a task performed in many places in \cite{z:geometric}. 

Among the many general features of balanced forcing extensions proved in \cite{z:geometric}, I will prove only one here.

\begin{proposition}
Let $P$ be a cofinally balanced poset. In the $P$-extension, every well-ordered sequence of elements of the ground model belongs itself to the ground model.
\end{proposition}

\noindent In particular, cofinally balanced forcing does not add any new sets of ordinals. While this behavior is impossible under the Axiom of Choice, in the choiceless context it is in fact quite common \cite{mathias:barren, dobrinen:barren}.

\begin{proof}
Let $p\in P$ be a condition and let $\tau$ be a name such that $p\Vdash\tau$ is a function from an ordinal $\ga$ to the ground model. Use the cofinal balance to find a set $x$ of ordinals such that $P, p, \tau\in\vd_x$, and there is an open set $O\subset P$ consisting of conditions below $p$ balanced over $x$. It will be enough to show that for each $\gb\in\ga$ there is an element $u\in\vd_{x}$ such that all conditions in $O$ force $\tau(\check\gb)=\check u$.

To see this, use the extension property to find sets $y_0\forkindep[x]y_1$ of ordinals, conditions $q_0\in\vd_{xy_0}$ and $q_1\in\vd_{xy_1}$ in $O$ and elements $u_0\in\vd_{xy_0}$ and $u_1\in\vd_{xy_1}$ such that $q_0\Vdash\tau(\check\gb)=\check u_0$ and $q_1\Vdash\tau(\check\gb)=\check u_1$. By the balance assumption on $O$, the conditions $q_0, q_1$ are compatible, so $u_0=u_1$ must hold. Thus, $u_0\in\vd_{xy_0}\cap\vd_{xy_1}$ and the latter intersection is equal to $\vd_x$ by Fact~\ref{fact1}(3). All conditions in $O$ then decide the statement $\tau(\check\gb)=\check u_0$ in the same way, and as $q_0$ decides it in the affirmative, so must all conditions in $O$.
\end{proof}

\section{Balanced forcing: examples}
\label{examplesection}

In this section, I produce several basic examples of (cofinally) balanced forcings with classification of their balanced sets. These examples are all present in \cite{z:geometric}; however, the arguments are much more natural in the new axiomatic setting.

\begin{example}
\label{ufexample1}
Let $P$ be the poset of all infinite subsets of $\gw$ ordered by modulo finite inclusion.  The poset is designed to add a nonprincipal ultrafilter on $\gw$.

\begin{enumerate}
\item The poset $P$ is balanced;
\item the regular open sets balanced over $x$ are classified by nonprincipal ultrafilters on $\gw$ in $\vd_x$.
\end{enumerate}
\end{example}

\begin{proof}
Start with (2). Let $x$ be a set of ordinals. For a nonprincipal ultrafilter $u\in\hvd_x$ on $\gw$ write $O_u=\{q\in P\colon q$ diagonalizes $u\}$. I will show that 

\begin{itemize}
\item each $O_u$ is regular open and balanced over $x$;
\item every regular open set balanced over $x$ is of the form $O_p$ for some nonprincipal ultrafilter $p\in\hvd_x$.
\end{itemize}

For the first item note that the set $O_u$ is nonempty since the ultrafilter $u$ contains only countably many subsets of $\gw$. In addition, $O_u\subset P$ is clearly open, regular, and in $\vd_x$. To prove the balance of the set $O_u$, suppose that $y_0\forkindep[x]y_1$ are sets of ordinals and $q_0\in \vd_{xy_0}$ and $q_1\in\vd_{xy_1}$ are conditions in the set $O_u$. If $q_0\cap q_1$ is infinite, then it is the required common lower bound of $q_0, q_1$. So, suppose towards a contradiction that the two sets have finite intersection; shrinking them if necessary, I may assume that they are actioally disjoint. By  Fact~\ref{fact1}(4), there is a set $a\in\vd_x$ such that $a\subset\gw$, $q_0\subseteq a$ and $q_1\cap a=0$. Now, consider the question whether $a\in u$ or not. If the answer is affirmative, $q_1$ should not be disjoint from it; if the answer is negative, then the complement of the set $a$ is in $u$ and $q_0$ should not be disjoint from it. In both cases, a contradiction is reached.

For the second item, if $O\subset P$ is a regular open set balanced over $x$, in view of Corollary~\ref{opencorollary} and the regularity assumption, it will be enough to find an ultrafilter $u\in\vd_x$ such that $O\cap O_u\neq 0$. To do that, use the extension axiom to find sets $y_0\forkindep[x]y_1$ of ordinals containing conditions $q_0$ and $q_1$ in $O$ respectively such that for every set $a\subset\gw$ in $\vd_x$, either $q_0\leq a$ or $q_0\leq\gw\setminus a$, and the same for subscript $1$. Since $q_0, q_1$ have a common lower bound  by the balanced assumption on $O$, the sets $\{a\in\vd_x\colon q_0\leq a\}$ and $\{a\in\vd_x\colon q_1\leq a\}$ must be equal. Their common value $u$ belongs to $\vd_x$ by Fact~\ref{fact1}(3), and by the choice of $q_0$ and $q_1$ it is an ultrafilter in $\vd_x$. Both conditions $q_0, q_1$ then witness the fact that $O\cap O_u\neq 0$.

(1) now immediately follows as whenever $x$ is a set of ordinals, in $\hvd_x$ the axiom of choice holds and every infinite subset of $\gw$ is contained in some nonprincipal ultrafilter on $\gw$.
\end{proof}

\begin{example}
\label{linearexample}
Let $V$ be a vector space over a well-orderable field $C$. Let $P$ be the poset of well-orderable linearly independent subsets of $V$, ordered by reverse extension. The poset is designed to add a Hamel basis of $V$ over $C$.

\begin{enumerate}
\item The poset $P$ is balanced;
\item if $x$ is a set of ordinals such that $C, V\in\vd_x$, then the regular open subsets of $P$ balanced over $x$ are classified by maximal linearly independent subsets of $V\cap\vd_x$ in $\vd_x$.
\end{enumerate}
\end{example}

\begin{proof}
To argue for (2), let $x$ be a set of ordinals such that $V, C\in\vd_x$. Note that $C\subset\vd_x$ holds by Fact~\ref{fact1}(2) and $V\cap\vd_x$ is again a vector space over $C$. For each inclusion-maximal linearly independent set $p\subset V\cap\vd_x$ in $\vd_x$ define $O_p=\{q\in P\colon q\leq p\}$. Then

\begin{itemize}
\item for each inclusion-maximal linearly independent subset $p\subset V\cap\vd_x$ in $\vd_x$, the set $O_p$ is regular open and balanced over $x$;
\item every regular open set balanced over $x$ is of this form.
\end{itemize}

To prove the first item, the regularity of $O_p$ is left to the reader. Suppose that $y_0, y_1$ are sets of ordinals such that $y_0\forkindep[x]y_1$. Suppose that $q_0\in\vd_{xy_0}$ and $q_1\in\vd_{xy_1}$ are elements of the open set $O$; I must prove that they have a common lower bound, meaning that their union is a linearly independent set. Suppose towards a contradiction that this fails, and there is a nontrivial linear combination $b$ of elements of $q_0\cup q_1$ giving zero. Use Fact~\ref{fact1}(2) to conclude that $q_0\subset\vd_{xy_0}$ and $q_1\subset\vd_{xy_1}$. Reorganizing the combination $b$, one can obtain an equality $b_0=b_1$, where $b_0$ contains only terms of $b$ in $\vd_{xy_0}$ and $b_1$ contains only terms in $\vd_{xy_1}$. Then the common value of $b_0$ and $b_1$ belongs to $\vd_x$ by the initial independence assumption on $y_0, y_1$ and Fact~\ref{fact1}(3). However, that common value is already expressed as a linear combination of elements of $p$. The linear independence of each of $q_0$ and $q_1$ then shows that the linear combinations $b_0$ and $b_1$ are the same, and the initial combination $b$ must have been trivial as desired.

To prove the second item, suppose that $O\subset P$ is a regular open set balanced over $x$. It will be enough to find a basis $p\in\vd_x$ of $V\cap\vd_x$ such that $O\cap O_p\neq 0$, as by Corollary~\ref{opencorollary} it then follows that $O=O_p$. Just use the extension property of $\forkindep$ to find sets of ordinals $y_0, y_1$ and conditions $q_0, q_1\in O$ such that the linear hull of both $q_0$ and $q_1$ contains $V\cap\vd_x$, $q_0\in\vd_{xy_0}$, $q_1\in\vd_{xy_1}$, and $y_0\forkindep[x]y_1$. By the balance of the set $O$, $q_0, q_1$ are compatible. Thus, for each $v\in V\cap \vd_x$, the linear combination $b_0$ of elements of $q_0$ which yields $v$ must be the same as the linear combination $b_1$ on the $q_1$ side. Since $b_0=b_1\in\vd_{xy_0}\cap\vd_{xy_1}$, Fact~\ref{fact1}(3) shows that these linear combinations belong to $\vd_x$, so must use only elements of $\vd_x$. It follows that the linear hull of $q_0\cap\vd_x=q_1\cap\vd_x$ is equal to $V\cap\vd_x$. By Fact~\ref{fact1}(3) again, the set $p=q_0\cap\vd_x=q_1\cap\vd_x$ belongs to $\vd_x$. The conditions $q_0, q_1$ both show that $O\cap O_p\neq 0$, completing the proof of the second item.

Now, (2) of the proposition follows. To prove (1), suppose that $x$ is a set of ordinals such that $V, C\in\vd_x$, and $p\in P\cap\vd_x$ is a condition. By Fact~\ref{fact1}(2), $p\subset\vd_x$ holds. The canonical well-ordering of $\vd_x$ provides a way of extending $p$ to a linearly independent subset of $V\cap \vd_x$ which is iclusion-maximal and in $\vd_x$. The above work then provides a set balanced over $x$ consisting of conditions stronger than $p$.
\end{proof}

\begin{example}
\label{loexample1}
Let $A$ be an arbitrary set. Let $P$ be the poset of all linear orderings on well-orderable subsets of $A$, ordered by reverse inclusion. The poset is designed to add a linear ordering of the set $A$.

\begin{enumerate}
\item the poset $P$ is balanced;
\item for every set $x$ of ordinals such that $A\in\vd_x$, the regular open sets balanced over $x$ are classified by linear orderings on $A\cap\vd_x$ which belong to $\vd_x$.
\end{enumerate}
\end{example}

\begin{proof}
First, argue for (2). Let $x$ be a set of ordinals such that $A\in\vd_x$. For each linear ordering $p\in\vd_x$ on $A\cap\vd_x$, let $O_p=\{q\in P\colon p\subseteq q\}$. Then 

\begin{itemize}
\item for every linear ordering $p\in\vd_x$, the set $O_p\subset P$ is regular open and balanced over $x$;
\item every regular open set balanced over $x$ is of this form.
\end{itemize}

To prove the first item, it is clear that every set $O_p$ is open, regular, and definable from $x$. To prove the balance of the set $O_p$, let $y_0, y_1$ be sets of ordinals such that $y_0\forkindep[x]y_1$ and $q_0\in\vd_{xy_0}$ and $q_1\in\vd_{xy_1}$ be conditions in $O_p$. It will be enough to show that $q_0$ and $q_1$ agree on the intersection of their domains, because such linear orderings can be amalgamated to a larger one. For this, first note that $\dom(q_0)\subset\vd_{xy_0}$ and $\dom(q_1)\subset\vd_{xy_1}$ by Fact~\ref{fact1}(2), and $\vd_{xy_0}\cap\vd_{xy_1}=\vd_x$ by Fact~\ref{fact1}(3). So, $\dom(q_0)\cap\dom(q_1)=\vd_x\cap A$, and the two linear orders agree on this set, being equal to $p$ on it.

For the second item, suppose that $O\subset P$ is a regular open set balanced over $x$; in view of Corollary~\ref{opencorollary} and the regularity assumption, it is enough to find a linear ordering $p$ on $A\cap\vd_x$ such that $O\cap O_p\neq 0$. Find sets $y_0\forkindep[x]y_1$ of ordinals such that $\vd_{xy_0}$ and $\vd_{xy_1}$ contain conditions $q_0$ and $q_1$ in $O$ respectively such that $A\cap\vd_x\subseteq\dom(q_0)$ and $A\cap\vd_x\subseteq\dom(q_1)$ both hold. Now, since $q_0$ and $q_1$ are compatible by the balance assumption on $O$, it must be the case that $q_0\restriction A\cap\vd_x=q_1\restriction A\cap\vd_x$. This last object, call it $p$, must belong to $\vd_x$ by Fact~\ref{fact1}(3). Now, $p\in\vd_x$ is a linear ordering on $A\cap\vd_x$, and both $q_0$ and $q_1$ show that $O\cap O_p\neq 0$ as required.

(1) now follows. Suppose that $x$ is a set of ordinals such that $A\in\vd_x$ holds, and let $p\in\vd_x\cap P$ be a condition. By Fact~\ref{fact1}(3), $p\subseteq\vd_x$ holds, and one can use the standard well-ordering of $\vd_x$ to extend $p$ to a linear ordering on all of $\vd_x\cap A$ which belongs to $\vd_x$. This ordering yields a set balanced over $x$ consisting of conditions stronger than $p$.
\end{proof}

\noindent The last example is somewhat more demanding. Let $d\geq 1$ be a number, and let $\Gamma_d$ be the graph on $\mathbb{R}^d$ connecting points of rational Euclidean distance. Komj{\' a}th \cite{komjath:rn} proved that each of these graphs has countable chromatic number in ZFC. There is a natural balanced poset adding a coloring of $\Gamma_d$. To state the definitions, fix the dimension $d$. A (partial) \emph{neighborhood coloring} is a (partial) function $c$ on $\mathbb{R}^d$ which to each point in its domain assigns one of its basic open neighborhoods in such a way that if $u, v$ are $\Gamma_d$-connected points in the domain of $c$, then either $u\notin c(v)$ or $v\notin c(u)$ holds.

\begin{example}
\label{algebraicexample}
The poset $P_d$ consists of conditions $p$ such that there is a countable real closed subfield $\supp(p)\subset\mathbb{R}$ such that $p$ is a neighborhood coloring with $\dom(p)=\supp(p)^d$. The ordering is defined by $q\leq p$ if $p\subseteq q$ and for every point $u\in\dom(p)$ and $v\in\dom(q)\setminus\dom(p)$, if $u\in c(v)$ then $u$ is not $\Gamma_d$-connected to $v$. The poset adds a countable $\Gamma_d$-coloring.

\begin{enumerate}
\item $P_d$ is balanced;
\item whenever $x$ is a set of ordinals, the sets balanced over $x$ are classified by neighborhood colorings of $\mathbb{R}^d\cap\vd_x$ in $\vd_x$.
\end{enumerate}
\end{example}

\begin{proof}
The whole argument relies on the following easy claim:

\begin{claim}
Let $C\subset\mathbb{R}$ be a real closed subfield and $v\in\mathbb{R}^d\setminus C^d$ be any point. Then there is a positive real number $\eps>0$ such that no element $u\in C^d$ which is closer than $\eps$ to $v$ is $\Gamma_d$-connected to $v$.
\end{claim}

\begin{proof}
If not, there would be points $u_n\in C^d$ and positive rationals $q_n$ such that $q_n$ is the distance between $u_n$ and $v$, and $\lim_nq_n=0$. Let $S_n$ be the sphere around $u_n$ of radius $q_n$. Then $S_n\subset\mathbb{R}^d$ is an algebraic set. Observe that $\bigcap_nS_n=\{v\}$ and that by the Hilbert Basis Theorem there is a number $m\in\gw$ such that $\bigcap_{n\in m}S_n=\{v\}$. This means that the point $v\in\mathbb{R}^d$ is algebraic over the set $\{u_n\colon n\in m\}\subset C^d$ and so is in $C^d$ by the initial assumption on $C$. This contradicts the initial assumption on $v$.
\end{proof}

\noindent Now, let $x$ be any set of ordinals and $p\in\vd_x$ be a neighborhood coloring with $\dom(p)=\mathbb{R}^d$.

\begin{itemize}
\item the set $O_p=\{q\in P_d\colon q\leq p\}$ is regular open and  balanced over $x$;
\item every regular open subset of $P_d$ is of this form.
\end{itemize}

For the first item, suppose that $y_0\forkindep[x]y_1$ and $q_0\in\vd_{xy_0}$ and $q_1\in\vd_{xy_1}$ are conditions stronger than $p$; I must find a common lower bound of $q_0, q_1$. Let $C\subset\mathbb{R}$ be a countable real closed field containing both $\supp(q_0)$ and $\supp(q_1)$ as a subset. Let $c$ be a neighborhood assignment on $C^d$. Let $r$ be a function defined on $C^d$ as follows: $q_0\cup q_1\subseteq r$, and for each $v\in\dom(r)\setminus \dom(q_0\cup q_1)$ let $r(v)$ be some basic open neighborhood of $v$ which is a subset of $c(v)$ and contains no points in $\dom(q_0)\cup\dom(q_1)$ which are $\Gamma_d$-connected to $v$. An important point is that such a neighborhood exists by the claim.

The first item will be proved if I show that $r\leq q_0, q_1$ holds. This is done by inspection of several cases. I will only show that if $u_0\in\dom(q_0)\setminus\vd_x$ and $u_1\in\dom(q_1)\setminus\vd_x$ are $\Gamma_d$-connected, then $u_0\notin q_1(u_1)$ and $u_1\notin q_0(u_0)$--the only point where the independence relation on $y_0, y_1$ is used. Suppose towards a contradiction that e.g. $u_0\in q_1(u_1)$ holds. Write $\gd$ for the Euclidean distance between $u_0, u_1$, and $S$ for the sphere around $u_1$ of radius $\gd$. Since $u_0\in S$, the independence assumption on $y_0, y_1$ shows that there is an algebraic set $T$ in $\vd_x$ such that $u_0\in T\subseteq S$ (see Proposition~\ref{alg} for a proof). Then $q_1(u_1)$ must be disjoint from $T$ as $q_1\leq p$ holds; in particular, $u_0\notin q_1(u_1)$ as desired.

For the second item, let $O\subseteq P_d$ be a nonempty regular open subset in $\vd_x$. It will be enough to find a neighborhod coloring $p\in\vd_x$ with domain $\mathbb{R}^d\cap\vd_x$ such that $O_p\cap O\neq 0$. To this end, use the extension property of $\forkindep$ to find sets $y_0, y_1$ of ordinals and conditions $q_0, q_1\in O$ such that $\mathbb{R}^d\cap\vd_x\subseteq\dom(q_0)$, $\mathbb{R}^d\cap\vd_x\subseteq\dom(q_1)$, $q_0\in\vd_{xy_0}$, $q_1\in\vd_{xy_1}$, and $y_0\forkindep[x]y_1$. The two conditions are compatible by the balance assumption on $O$, so the sets $q_0\restriction\vd_x$ and $q_1\restriction\vd_x$ must be the same. By Fact~\ref{fact1}(3), their common value $p$ belongs to $\vd_x$, and it is a neighborhood coloring on $\mathbb{R}^d\cap\vd_x$. It also must be the case that $q_0\leq p$, since otherwise the conditions $q_0, q_1$ would fail to be compatible; similarly, $q_1\leq p$. Thus, the conditions $q_0, q_1$ witness the fact that $O_p\cap O\neq 0$ as required.

Now, (2) of the proposition follows. For (1), suppose that $p\in P_d$ is a condition and $x$ is a set of ordinals such that $p\in\vd_x$; I must find a neighborhood coloring $q\in\vd_x$ such that $q\leq p$ and $\dom(q)=\mathbb{R}^d\cap\vd_x$. To this end, first find an arbitrary neighborhood coloring $r\in\vd_x$ with $\dom(r)=\mathbb{R}^d$. Such a coloring exists as a consequence of AC in $\hvd_x$ by a result of \cite{z:twographgames}. Now, note that $\dom(p)\subseteq\mathbb{R}^d\cap\vd_x$ by Fact~\ref{fact1}(3). For every point $v\in(\mathbb{R}^d\cap\vd_x)\setminus\dom(p)$ let $o(v)$ be the first basic open neighborhood of $v$ in some fixed enumeration such that $o(v)$ contains no points in $\dom(p)$ which are $\Gamma_d$-related to $v$. Such a neighborhood exists by the claim. Finally, let $q$ be the map defined by $q(v)=p(v)$ if $v\in\dom(p)$ and $q(v)=r(v)\cap o(v)$ for $v\in (\mathbb{R}^d\cap\vd_x)\setminus\dom(p)$. An inspection of the definitions shows that $q\leq p$ has the required properties.
\end{proof}

\section{Independence relations}
\label{independencesection}

Many consistency proofs in \cite{z:geometric, z:vitali, z:triangles, z:distance, z:ngraphs} can be re-organized so as to use variations of the master independence relation $\forkindep$. These variations are interesting in their own right; this section treats them in their generality.

\begin{definition}
Let $d\geq 2$ be a number. A class $d+1$-ary relation $\perp$ on sets of ordinals is an \emph{independence relation} if it satisfies the following properties, for every set $x$ of ordinals such that $\perp$ is definable from $x$.

\begin{enumerate}
\item (nontriviality) if $\bar z$ is the $d$-tuple all of whose entries are equal to $x$ then $\perp_x\bar z$;
\item (monotonicity) suppose that $i\in d$ is an index, $\bar y$ and $\bar z$ are $d$-tuples which agree on all entries except possibly at $i$, and $\perp_x\bar y$ holds. If $\bar z(i)\in\vd_{x\bar y(i)}$ then $\perp_x\bar z$ holds;
\item (transitivity) suppose that $i\in d$ is an index, $\bar y$ and $\bar z$ are $d$-tuples which agree on all entries except possibly at $i$, and $\perp_x\bar y$ holds. If $\bar z(i)\forkindep[x\bar y(i)]\bar y\restriction d\setminus\{i\}$ then $\perp_x\bar z$.
\end{enumerate}

\noindent In addition, the relation may or may not have the following features:

\begin{enumerate}
\item[(4)] (symmetry) if $\bar y$ and $\bar z$ are $d$-tuples which are permutations of each other, then $\perp_x\bar y$ is equivalent to $\perp_x\bar z$;
\item[(5)] (strong transitivity) suppose that $i\in d$ is an index, $\bar y$ and $\bar z$ are $d$-tuples which agree on all entries except possibly at $i$, and $\perp_x\bar y$ holds. If $\perp_{x\bar y(i)}\bar z$ then $\perp_x\bar z$.
\end{enumerate}

\noindent In case of a $2+1$-ary relation, the notation will be $y_0\perp_xy_1$.
\end{definition}

\noindent In the common case of $2+1$-ary symmetric relations, transitivity implies that $y_0\forkindep[x]y_1$ implies $y_0\perp_xy_1$. On the other hand, if $y_0\forkindep[x]y_1$ implies $y_0\perp_xy_1$, then strong transitivity implies transitivity. These observations are left to the kind reader.

The following is the requisite strengthening of the notion of balance.

\begin{definition}
Let $P$ be a poset and $\perp$ be a $d+1$-ary independence relation.

\begin{enumerate}
\item if $x$ is a set of ordinals such that $\perp$ is definable from $x$ and $O\subset P$ is a nonempty open set in $\vd_x$, say that $O$ is \emph{$\perp$-balanced} over $x$ if for every $d$-tuple $\bar y$ of sets of ordinals such that $\perp_x\bar y$ and any $d$-tuple $\bar p$ of conditions such that $\forall i\in d\ \bar p(i)\in O\cap\vd_{x\bar y(i)}$, the tuple $\bar p$ has a common lower bound in $P$;
\item $P$ is \emph{$\perp$-balanced} if for every set $x$ of ordinals such that $P\in\vd_x$ and $\perp$ is definable from $x$, for every condition $p\in P\cap\vd_x$ there is a set $O\subset P$ $\perp$-balanced over $x$ consisting of conditions stronger than $p$;
\item $P$ is \emph{cofinally $\perp$-balanced} if for every condition $p\in P$ and every set $x$ of ordinals, there is a set $y$ of ordinals such that $P, p\in\vd_y$ and $\perp$ is definable from $y$, and there is a set $O\subset P$ $\perp$-balanced over $y$ such that $O$ consists of conditions stronger than $p$.
\end{enumerate}
\end{definition}

\noindent Thus, the larger the independence relation, the smaller the class of its (cofinally) balanced posets. In all natural cases, either all sets balanced over $x$ are $\perp$-balanced, or none of them are. The most useful feature of $\perp$-balance is the following proposition, which is the main reason for insisting on the transitivity property for weak independence relations.

\begin{proposition}
\label{extensionproposition2}
Let $\perp$ be a $d+1$-ary independence relation. Let $\langle P, \leq\rangle$ be a poset and $x$ a set of ordinals such that $\langle P, \leq\rangle\in\vd_x$. Let $O\subset P$ be a nonempty open set in $\vd_x$. The following are equivalent:

\begin{enumerate}
\item $O$ is $\perp$-balanced over $x$;
\item whenever $\bar y$ is a $d$-tuple such that $\perp_x\bar y$ holds, and $\bar O'$ is a $d$-tuple of nonempty open subsets of $O$ such that $\forall i\in d\ \bar O'(i)\in\vd_{x\bar y(i)}$, then $\bigcap_{i\in d}\bar O'(i)\neq 0$.
\end{enumerate}
\end{proposition}

\begin{proof}
The implication (2)$\to$(1) is immediate by considering the open sets $\bar O'(i)=\{q\in P\colon q\leq\bar p(i)\}$ for each $i\in d$. For (1)$\to$(2), suppose that $O$ is $\perp$-balanced over $x$ and $\bar y, \bar O'$ are as in (2). By recursion on $i\leq d$ build $d$-tuples $\bar z_i$ so that 

\begin{itemize}
\item $\bar z_0=\bar y$, for every positive $i\leq d$ $\perp_x\bar z_i$ holds, and $\bar z_{i-1}(j)\in\vd_{x\bar z_{i}(j)}$ holds for every $j\in d$;
\item for each positive $i\leq d$, there is an element $\bar p(i-1)\in \bar O'\cap\vd_{x\bar z_i(i-1)}$.
\end{itemize}

\noindent To perform the recursion step, suppose that $i\in d$ and $\bar z_i$ has been obtained. Use the extension property of $\forkindep$ to find a set $w$ of ordinals such that $w\forkindep[x\bar z_i(i)](z_i\restriction d\setminus\{i\})$ and there is an element $\bar p(i)\in \bar O'(i)\cap \vd_{x\bar z_i(i)w}$. Obtain $\bar z_{i+1}$ from $\bar z_i$ by replacing $\bar z_i(i)$ with $\bar z_i(i)w$. The transitivity implies that$\perp_x\bar z_{i+1}$ holds as desired.

In the end, by $\perp_x\bar z_d$ and the balance assumption on $O$, the conditions $\bar p(i)$ for $i\in d$ have a common lower bound, which is the desired element of $\bigcap_{i\in d}\bar O'(i)$.
\end{proof}

Below, I go through a list of basic independence relations. Each section contains the following subsections: (1) definition of the $\perp$ relation, verification that it indeed is a weak independence relation, with basic examples of cofinally $\perp$-balanced forcings; (2) production of some independent generic extensions; and (3) basic preservation theorems for the class of cofinall $\perp$-balanced forcings, finally yielding some independence results. The whole scheme is somewhat parallel to variations of proper forcing in its flexibility.

\section{Placidity}
\label{0section}

\noindent The most permissive among useful weak independence relation is associated with a class of placid partial orders as isolated in essence in \cite[Section 9.3]{z:geometric}.

\subsection{$\perp^0$ definition}

\begin{definition}
$\perp^0$ is the relation defined by $y_0\perp^0_xy_1$ if $\hvd_{xy_0}\cap\hvd_{xy_1}=\hvd_x$.
\end{definition}

\begin{proposition}
$\perp^0$ is an independence relation with symmetry and strong transitivity.
\end{proposition}

\begin{proof}
Monotonicity and symmetry are obvious. If $y_0\forkindep[x]y_1$ then $y_0\perp_xy_1$ holds by Fact~\ref{fact1}(3). Thus, it is only necessary to verify strong symmetry. Suppose that $y_0\perp_x y_1$ and $z\perp_{xy_0}y_1$ holds; I must show that $z\perp_xy_1$ follows. To this end, let $A\in\hvd_{xz}\cap\hvd_{xy_1}$ be any set. By the latter assumption, $A\in\hvd_{xy_0}$ holds, and by the former assumption, $A\in\hvd_x$ follows as desired.
\end{proof}

\begin{example}
\label{hamelexample}
(adding a Hamel basis to $\mathbb{R}$)
Let $P$ be the partial ordering of countable subsets of $\mathbb{R}$ linearly independent over $\mathbb{Q}$, ordered by reverse inclusion. Then $P$ is $\perp^0$-balanced.
\end{example}

\begin{proof}
By inspection of proof of Example~\ref{linearexample}. There, the only feature of $y_0\forkindep[x]y_1$ which is used in the case of the vector space $\mathbb{R}$ over $\mathbb{Q}$ is that $\mathbb{R}\cap\vd_{xy_0}\cap\vd_{xy_1}=\mathbb{R}\cap\vd_x$.
\end{proof}

\subsection{$\perp^0$-independent extensions}

\noindent The most interesting way of generating pairs of $\perp^0$-independent extensions comes from Hjorth's concept of turbulent actions of Polish groups.

\begin{definition}
Let $G$ be a Polish group acting continuously on a Polish space $Y$.

\begin{enumerate}
\item For open sets $U\subset G$ containing the unit and $O\subset Y$, write $E_{UO}$ for the smallest equivalence relation on $O$ containing all pairs $\langle y, gy\rangle$ such that $y\in O$, $g\in U$, and $gy\in O$;
\item the action is \emph{turbulent} if for all nonempty sets $U\subset G$ containing the unit and $O\subset Y$, the relation $E_{UO}$ contains a somewhere dense equivalence class.
\end{enumerate}
\end{definition}

\begin{example}
\label{c0example}
Let $G$ be the separable Banach space of all sequences in $\mathbb{R}^\gw$ converging to zero, with the topology induced by the maximum norm. It acts on $\mathbb{R}^\gw$ by coordinatewise addition. The action is easily check to be turbulent.
\end{example}

\begin{proposition}
\label{0prop}
Let $G\acts Y$ be a turbulent action. Let $x$ be a set of ordinals such that $\vd_x$ contains $G, Y$, and the action. Let $g\in G$ and $y\in Y$ be points mutually Cohen-generic over $x$. Then $y\perp^0_xg\cdot y$ holds.
\end{proposition}

\begin{proof}
Let $x$ be a set of ordinals such that $G, Y$, and the action belongs to $\vd_x$. The following claim is proved by a routine repeated use of Proposition~\ref{mutualproposition}.

\begin{claim}
Let $n\geq 1$ be a natural number, and $z\in Y$ and $\bar h\in G^n$ be points mutually Cohen-generic over $x$. Then $\prod(\bar h\restriction n-1)\cdot z\in Y$ and $\bar h(n-1)\in G$ are mutually Cohen-generic over $x$.
\end{claim}
 
Now, suppose that $g\in G$ and $y\in Y$ are points mutually Cohen-generic over $x$. Suppose towards a contradiction that the conclusion of the proposition fails, so $\hvd_{xy_0}\cap\hvd_{xy_1}\neq\hvd_x$. An $\in$-minimal set witnessing the inequality belongs to both $\hvd_{xy_0}$ and $\hvd_{xy_1}$ and it is a subset (but not an element) of $\hvd_x$. Let $F_0, F_1$ be $\vd$-functions such that $F_0(x, y_0)=F_1(x, y_1)$ is such a set. Let $B\subseteq G\times Y$ be the $\vd_x$ set of all pairs $\langle h, z\rangle$ such that $F_0(x, z)=F_1(x, hz)$ is a subset, but not an element, of $\hvd_x$. The pair $\langle g, y\rangle$ belongs to $B$, so $B$ is non-meager. Since all sets have the Baire property, there must be nonempty open sets $U\subset G$ and $O\subset Y$ such that $B$ is comeager in $U\times O$. Shifting the set $U$ by an element of $G\cap\hvd_x$ if necessary, I can reduce the proof to the case where $1\in U$. Shrinking $U$ if necessary, assume that $U=U^{-1}$.

Use the turbulence assumption to find a basic open set $O'\subset O$ in which an $E_{UO}$-class is dense. There must be a set $a\in\hvd_x$ such that the set $C_a=\{y\in O'\colon a\in F_0(x, y)\}$ is neither meager nor comeager in $O'$. If this were not the case,  write $b=\{a\in\hvd_x\colon C_a$ is comeager in $O'\}$, find a pair $\langle h, z\rangle\in U\times G'$ Cohen-generic over $x$, and conclude that $F_0(x, z)=b\in\vd_x$. This would contradict the choice of the sets $U$ and $O$.

Pick a set $a$ as in the previous paragraph. Note that the set $B_a$ has the Baire property. By the definition of the equivalence relation $E_{UO}$, there must be a finite sequence $\langle y_i\colon i\leq n, h_i\colon <n\rangle$ such that $n$ is even, $B_a$ is co-meager in some neighborhood of $y_0$, it is meager in some neighborhood of $y_n$, for each $i<n$ $y_i\in O$, $h_i\in U$, and $h_iy_i=y_{i+1}$. By the continuity of the action, there must be open sets $O_i\subseteq O'$ for $i\leq n$, $U_i\subseteq U$ for $i<n$, such that $B_a$ is co-meager in $O_0$, meager in $O_n$, and for all $i\in n$ $U_iO_i\subset O_{i+1}$ holds.

Now, let $z_0\in O_0$ and $g_i\in U_i$ be points mutually Cohen-generic over $x$, and by recursion on $i$ define $z_{i+1}=g_iz_i$. Observe that $z_i\in O_i$ and $g_i\in U$ are points mutually Cohen-generic over $x$ by the claim. By the initial choice of the sets $U$ and $O$, the pairs $\langle g_i, z_i\rangle$ and $\langle g_i^{-1}, z_i\rangle$ are both elements of the set $B$. 

Now, by the definitions, $a\in F_0(x, z_0)$ and $a\notin F_0(x, z_n)$, so there must be an even number $i\in n$ such that $a\in F_0(x, z_i)\setminus F_0(x, z_{i+2)}$. However, $F_0(x, z_i)=F_1(x, g_iz_i)=F_1(x, g_{i+1}^{-1}z_{i+2})=F_0(x, z_{i+2})$ must hold by the definition of the set $B$. A contradiction.
\end{proof}

\subsection{$\perp^0$ consistency results}

\begin{theorem}
Let $G\acts Y$ be a generically turbulent continuous Polish group action. In every cofinally $\perp^0$-balanced extension, every $G$-invariant subset of $Y$ is either meager or co-meager.
\end{theorem}

\begin{proof}
Let $P$ be a cofinally $\perp^0$-balanced poset. If the conclusion were to fail, there would have to be a condition $p\in P$ and a $P$-name $\tau$ for a subset of $Y$ such that $p\Vdash\tau$ is $G$-invariant and both $\tau$ and $Y\setminus\tau$ are non-meager. Let $x$ be any set of ordinals defining all objects mentioned so far, such that there is a $\perp_0$-balanced set $O\subset P$ consisting only of conditions stronger than $p$. Both sets $A_0=\{y\in Y\colon$ there is a condition $q\in O$ forcing $\check y\in\tau\}$ and $A_1=\{y\in Y\colon$ there is a condition $q\in O$ forcing $\check y\notin\tau\}$ are $\vd_x$ and must be non-meager since any condition in $O$ forces $\tau\subseteq\check A_0$ and $Y\setminus \tau\subseteq\check A_1$. Use the Baire axiom to find basic open sets $O_0$ and $O_1\subset Y$ such that $A_0$ is co-meager in $O_0$ and $A_1$ is comeager in $O_1$. Thinning out the set $O_0$ if necessary, I may find a nonempty open set $U\subset G$ such that $U\cdot O_0\subset O_1$. Let $g\in U$ and $y\in O_0$ be mutually Cohen-generic points over $x$.

Then $y\in A_0$ and $g\cdot y\in A_1$, and $y\perp^0_xg\cdot y$ holds by Proposition~\ref{0prop}. The sets $\{q\in O\colon q\Vdash y\in\tau\}$ and $\{q\in O\colon q\Vdash g\cdot y\notin\tau\}$ are both nonempty open subsets of $O$ in $\vd_{xy}$ and $\vd_{xg\cdot y}$ respectively. By Proposition~\ref{extensionproposition2}, they have a nonempty intersection. Any condition in the intersection forces that $\tau$ is not $G$-invariant, a contradiction.
\end{proof}

\begin{corollary}
\label{linocorollary}
Let $G\acts Y$ be a generically turbulent continuous Polish group action.  The orbit space cannot be linearly ordered in any cofinally $\perp^0$-balanced extension.
\end{corollary}

\begin{proof}
It is easy to check that the product action $G^2\acts Y^2$ is turbulent. Let $\leq$ be a linear ordering of the orbit space, viewed as a $G^2$-invariant subset of $Y^2$. Let $\pi\colon Y^2\to Y^2$ be the homeomorphism exchanging the two coordinates of the product. Then $\geq=\pi''\leq$, and $\geq\cap\leq$ is exactly the $G$-orbit equivalence relation, which is meager, having all classes meager. A situation of this type cannot occur in a model in which every $G^2$-invariant subset of $Y^2$ is meager or co-meager: both $\leq$ and $\geq$ have the same category, as they cover the whole space they would have to be co-meager, but they intersect in a meager set.
\end{proof}

\begin{corollary}
It is consistent with ZF+DC+there is a Hamel basis for $\mathbb{R}$ over $\mathbb{Q}$ that the quotient space of the $c_0$ equivalence relation on $\cantor$ is not linearly orderable.
\end{corollary}

\begin{proof}
The poset adding a Hamel basis is $\perp^0$-balanced by Example~\ref{hamelexample}. The $c_0$ equivalence relation on $\cantor$ is induced as an orbit equivalence relation of turbulent Polish group action by Example~\ref{c0example}. An application of Corollary~\ref{linocorollary} completes the proof.
\end{proof}

\section{OD-placidity}
\label{1section}

Examples of Section~\ref{0section} produce pairs $y_0\perp^0_xy_1$ such that $y_0, y_1$ come from the same $E$-equivalence class, where $E$ is a turbulent equivalence relation and the class does not belong to $\hvd_x$. In consequence, $\vd_{xy_0}\cap\vd_{xy_1}\neq\vd_x$. This section produces an independence relation which eliminates this feature.

\subsection{$\perp^1$-definition}

\begin{definition}
\label{perp1definition}
$\perp^1$ is the relation defined by $y_0\perp^1_xy_1$ if $\vd_{xy_0}\cap\vd_{xy_1}=\vd_x$.
\end{definition}

\noindent It is immediate that $\perp^1$ is a subset of $\perp^0$.

\begin{proposition}
$\perp^1$ is an independence relation with symmetry and strong transitivity.
\end{proposition}

\begin{proof}
Nontriviality, monotonicity, and symmetry are obvious. $y_0\forkindep_xy_1$ implies $y_0\perp^1_xy_1$ by Fact~\ref{fact1}(3). It only remains to prove strong transitivity. Suppose that $x, y_0, y_1, z$ are such that $y_0\perp^1_xy_1$ and $z\perp^1_{xy_0}y_1$ both hold. To show that $y_0z\perp^1_xy_1$ holds, let $A\in\vd_{xy_0z}\cap\vd_{xy_1}$ be any set and work to show that $A\in\vd_x$ holds. The second assumption implies $A\in\vd_{xy_0}$, and then the first assumption shows that $A\in\vd_x$ as required. The proof is complete.
\end{proof}

\begin{example}
Let $A$ be any set. Let $P=\{p\colon p$ is a linear ordering on a well-orderable set $\dom(p)\subseteq A\}$. The ordering is that of reverse inclusion. Then $P$ is $\perp^1$-balanced.
\end{example}

\begin{proof}
By inspection of the proof of Example~\ref{loexample1}. The only feature of $y_0\forkindep[x]y_1$ used in it is $\vd_{xy_0}\cap\vd_{xy_1}=\vd_x$.
\end{proof}

\subsection{$\perp^1$-independent extensions}

\noindent Producing interesting $\perp^1$-independent pairs of generic extensions is slightly more difficult than in the $\perp^0$ case. 

\begin{definition}
For a Polish group $G$, a natural number $d\geq 1$, and a tuple $\bar g\in G^d$, $\bar g!$ (the \emph{alternating} $\bar g$-\emph{factorial}) stands for the product $\bar g(d-1)^{\mp 1}\cdot \bar g(d-2)^{\pm 1}\cdot\dots\cdot\bar g(2)\bar g(1)^{-1}\bar g(0)\in G$. For a $d$-tuple $\bar U$ of nonempty subsets of $G$, $\bar U!$ denotes the set $\{\bar g!\colon \bar g\in\prod\bar U\}\subseteq G$.
\end{definition}

\begin{definition}
Let $G$ be a Polish group and $C\subset G$ be a nonempty closed set. Say that the set $C\subseteq G$ has \emph{bounded local generation} if there is $n\in\gw$ such that for every $n$-tuple $\bar U$ of nonempty relatively open subsets of $C$, the set $\bar U!$ has nonempty interior.
\end{definition}

\begin{example}
\label{sphereexample}
Let $d>1$ be a natural number, let $G=\mathbb{R}^d$ and let $C\subset G$ be the unit sphere. It is not difficult to see that $C$ has bounded local generation, with $n=2$ being sufficient.
\end{example}

\begin{example}
Let $G=\mathbb{Z}^\gw$ and let $C=A^\gw$ where $A$ is the set of all (positive) primes. Then $C$ has bounded local generation with $n=6$ being sufficient. To verify this, recall Vinogradov's theorem \cite{nathanson:bases}: every odd number larger than five is the sum of at most three primes. By a discussion of several cases, using the key identity $2+3=5$, it follows that every integer is a difference between two sums of three primes, or $G=C^6!$. Relativization to arbitrary open subsets of $C$ is left to the kind reader.
\end{example}

\begin{proposition}
\label{turbidproposition}
Suppose that $C\subseteq G$ is a set with bounded local generation. Let $x$ be a set of ordinals which defines the triple and countable bases for the topologies. Let $c\in C$ and $g\in G$ be points mutually Cohen-generic over $x$. Then $g\perp^1_xc\cdot g$ holds.
\end{proposition}

\begin{proof}
I will start with a general preliminary claim. Let $x$ be a set of ordinals such that $\vd_x$ contains both $G$ and $C$.

\begin{claim}
\label{toughclaim}
Let $n\geq 1$ be a number and $g\in G$ and $\bar c\in C^n$ be mutually Cohen-generic points over $x$. Then

\begin{enumerate}
\item the points $\bar (c\restriction n-1)! g\in G$ and $\bar c(n0\in C$ are mutually generic over $x$;
\item if $n$ witnesses bounded local generation of $C$ in $G$, then $g, \bar c!g\in G$ are points mutually Cohen-generic over $x$.
\end{enumerate}
\end{claim}

\noindent This is proved by routine repeated use of Proposition~\ref{mutualproposition}. Now, let $g\in G$ and $c\in C$ be mutually Cohen-generic over $x$. Suppose towards a contradiction that $\vd_{xg}\cap\vd_{x, c\cdot g}\neq\vd_x$, and find $\vd$ definable functions $F_0, F_1$ such that $F_0(x, g)=F_1(x, c\cdot g)\notin\vd_x$. Consider the set $B=\{\langle h, d\rangle\in G\times C\colon F_0(x, h)=F_1(x, d\cdot h)\notin\vd_x\}$. This is a $\vd_x$ set containing the Cohen-generic pair $\langle g, c\rangle$, so it must be non-meager. Since every set has the Baire property, ther must be basic open sets $U\subset G$ and $O\subset C$ such that $B$ is co-meager in $O\times U$. Use the continuity of the action to find nonempty open sets $O'\subset O$ and $\bar V(i)\subset U$ for $i\in n$ such that for each even $i$, $(\bar V\restriction i)!\cdot O'\subset O$. Let $h\in G$ and $\bar d$ be mutually generic elements of $Y$ and $C^n$ below the conditions $O'$ and $\bar V$ respectively. For each $i\leq n$ write $h_i=(\bar d\restriction i)!\cdot h$.  

It is clear from Claim~\ref{toughclaim}(1) that for each even $i\in n$, $h_i\in U$ and $\bar d(i)\in O$ are mutually Cohen-generic elements over $x$, so $\langle h_i, \bar d(i)\rangle\in B$ holds. Thus, $F_0(x, h_i)=F_1(x, h_{i+1})\notin\vd_x$ holds. Similarly, $h_{i+2}\in U$ and $\bar d(i+1)\in O$ are mutually generic elements as well, and $F_0(x, y_{i+2})=F_1(x, y_{i+1})$ holds. It follows that $F_0(x, h_0)=F_0(x, h_n)$ and this set does not belong to $\vd_x$. This, however, contradicts the mutual genericity of $h_0$ and $h_n$ over $x$ (Claim~\ref{toughclaim}(2)) and Fact~\ref{fact1}(3).
\end{proof}

\subsection{$\perp^1$ consistency results}

\begin{theorem}
Suppose that $P$ is a cofinally $\perp^1$-balanced forcing. In the $P$-extension, for every nonmeager set $A\subset \mathbb{R}^2$ there is a real number $\eps>0$ such that for all positive reals $\gd<\eps$ there are two points in $A$ at Euclidean distance exactly $\delta$.
\end{theorem}

\begin{proof}
Let $P$ be a cofinally $\perp^1$-balanced forcing. Let $\tau$ be a $P$-name and $p\in P$ be a condition such that $p\Vdash\tau\subset\mathbb{R}^2$ is a non-meager set. Use the assumption on $P$ to find a set $x$ of ordinals and a set $O\subset P$ which is balanced over $x$ and consists only of conditions stronger than $p$. Consider the set $A=\{z\in\mathbb{R}^2\colon\exists q\in O\ q\Vdash \check z\in\tau\}\in\vd_x$. This set must be non-meager by the original assumption on $p$ and $\tau$, and by the Baire axiom, there must be an open ball $B\subset\mathbb{R}^2$ such that $B\setminus A$ is meager. Let $\eps>0$ be a real number which is smaller than the radius of $B$. It will be enough to show that any condition in $O$ forces $\eps$ to work as in the theorem for $\tau$.

To see this, let $\gd<\eps$ be a positive real number and $q\in O$ be a condition. I must find points $z_0, z_1\in\mathbb{R}^2$ at distance $\gd$ and a condition $r\leq q$ such that $r\Vdash \check z_0, \check z_1\in\tau$. To do this, let $y$ be a set of ordinals such that $\gd, q\in\vd_{xy}$, and there is a set $O'\subset P$ which is $\perp^1$-balanced over $xy$ and consists only of conditions stronger than $q$. 
Let $C$ be the sphere around the origin of radius $\gd$. Find points $g\in B$ and $c\in C$ Cohen-generic over ${xy}$ such that $g\forkindep_{xy} c$ and both points $z_0=g$ and $z_1=z_0+v$ belong to the open ball $B$. Notice that both of these points are points in $\mathbb{R}^2$ Cohen-generic over ${xy}$, therefore over $x$ as well. There is a string of consequences to this observation:

\begin{itemize}
\item by the choice of the open neighborhood $B$, $z_0, z_1\in A$;
\item $y\forkindep_xz_0$ and $y\forkindep_xz_1$;
\item by the balance assumption on the set $O$, the sets $O_0=\{r\in O'\colon r\Vdash\check z_0\in\tau\}\in\vd_{xyz_0}$ and $O_1=\{r\in O'\colon r\Vdash\check z_1\in\tau\}\in\vd_{xyz_1}$ are nonempty.
\end{itemize}

\noindent Now, by Example~\ref{sphereexample}, $z_0\perp^1_{xy}z_1$ holds, and the balance assumption on the set $O'$ shows that $O_0\cap O_1\neq 0$. Any condition $r$ in the intersection has the required properties.
\end{proof}

\section{Interpolation in Artinian semilattices}
\label{artiniansection}

In \cite{z:triangles, z:krull}, I produced a line of ZF consistency results calibrated by dimensions of Euclidean spaces. The current section provides a rather flexible, simplified treatment of such consistency results. The dimension is measured by rank of distributive Artinian lattices.

\subsection{$\perp^L$ definition}

Start with recalling basic facts about Artinian semilattices. A meet semilattice $\langle L, \land\rangle$ is \emph{Artinian} if every strictly decreasing sequence in $L$ is finite. For such a semilattice, under DC, for every set $A\subset L$ there is a finite set $a\subseteq A$ such that $\bigwedge a\leq u$ for every $u\in A$. To see this, consider the set of all finite sets $b\subseteq A$ pre-ordered by $b_1\leq b_0$ if $\bigwedge b_1\leq\bigwedge b_0$ in $L$ holds. By DC, the associated quotient partial order either has a minimal element or an infinite strictly descending sequence. The infinite descending sequence is precluded by the Artinian assumption on $L$, and a minimal element is represented by a finite set $a\subseteq A$ with the requested property.

\begin{definition}
Let $\langle L, \land\rangle$ be an Artinian semilattice. Write $y_0\perp^L_xy_1$ if $L\in\vd_x$ and for all elements $u_0\leq u_1$ in $L\cap\vd_{xy_0}$ and $L\cap\vd_{xy_1}$ respectively there is an \emph{interpolant} $v\in \vd_x$ such that $u_0\leq v\leq u_1$, and similarly for $u_1\leq u_0$.
\end{definition}

\begin{proposition}
\label{alg}
$\perp^L$ is an independence relation with symmetry and strong transitivity.
\end{proposition}

\begin{proof}
The monotonicity and symmetry properties are obvious. Let $x$ be a set of ordinals such that $L\in\vd_x$; I must show that $y_0\forkindep_x y_1$ implies $y_0\perp^L_xy_1$.

Let $u_0\in \vd_{xy_0}$ and $u_1\in\vd_{xy_1}$ be arbitrary elements in $L$, and let $F_0$ and $F_1$ be $\vd$-functions such that $u_0=F_0(x, y_0)$ and $u_1=F_1(x, y_1)$. Assume for definiteness that $u_0\leq u_1$. Use the product property of $\forkindep$ to find $\vd_{x}$ sets $A_0$ and $A_1$ such that $y_0\in A_0$ and $y_1\in A_1$, and whenever $y'_0\in A_0$ and $y'_1\in A_1$ are $\forkindep_{x}$ independent, then $F_0(x, y'_0)\leq F_1(x, y'_1)$ holds. Let $v=\bigwedge\{F_1(x, y'_1)\colon y'_1\in A_1\}\in\vd_{x}$. It will be enough to show that $u_0\leq v\leq u_1$.

Now, $v\leq u_1$ is immediate from the definition of $w$. To see that $u_0\leq v$ holds, suppose towards a contradiction that it fails, so the set $B_0=\{y'_0\in A_0\colon F_0(x, y'_0)\not\leq v\}$ is nonempty. Use the Artinian property of the lattice $L$ to find a finite set $a_1\subset A_1$ such that $v=\bigwedge\{F_1(x, y'_1)\colon y'_1\in a_1\}$. Let $y'_0\in B_0$ be an element such that $y'_0\forkindep_{x}y'_1$ holds for all $y'_1\in b$. The choice of the sets $A_0$ and $A_1$ shows that $F_0(x, y'_0)\leq F_1(x, y'_1)$ holds for all $y'_1\in a_1$, so $F_0(x, y'_0)\leq v$ by the choice of the set $a_1$. This contradicts the definition of the set $B_0$.

All that remains is to show strong transitivity. Suppose that $y_0\perp^L_xy_1$ and $z\perp^L_{xy_0}y_1$ and argue for $y_0z\perp^L_xy_1$. Suppose for example that $u_0\in L\cap\vd_{xy_0z}$, $u_1\in L\cap \vd_{xy_1}$, and $u_0\leq u_1$; we need to find an interpolant in $\vd_x$. First of all, the second assumption yields an interpolant $w$ between $u_0$ and $u_1$ in $\vd_{xy_0}$. The second assumption then yields an interpolant between $w$ and $u_1$ in $\vd_x$. Strong transitivity follows.
\end{proof}

\begin{example}
\label{algd}
Let $d>0$ be a number, and let $L$ be the lattice of algebraic subsets of $\mathbb{R}^d$ with intersection. Let $P_d$ be the poset of Example~\ref{alg} for coloring the rational distance graph on $\mathbb{R}^d$ with countably many colors. Then $P_d$ is $\perp^L$-balanced.
\end{example}

\begin{proof}
By inspection of the proof of Example~\ref{alg}.
\end{proof}

\subsection{$\perp^L$-independent generic extensions}

The main point of the current section is learning to distinguish between different Artinian semilattices. In order to do that, recall a well-known notion of rank of Artinian distributive lattice. Suppose that $\langle L, \land, \lor\rangle$ is a distributive Artinian lattice. Call an element $u\in L$ \emph{irreducible} if for any $v_0, v_1\in L$, if $u\leq v_0\lor v_1$ then $u\leq v_0$ or $u\leq v_1$ holds. Every element $u\in L$ can be written in a unique way as $u=\bigwedge a$ where $a$ is a finite set of irreducible elements. One consequence of the uniqueness is that if $x$ is a set of ordinals such that $u\in\vd_x$, then $a\in\vd_x$ and by Fact~\ref{fact1}(2) $a\subset\vd_x$. The \emph{rank} of an Artinian lattice is the maximal length of strictly decreasing sequence of irreducible elements if such exists. It is well-known that the lattice of algebraic subsets of $\mathbb{R}^n$ has rank $n$.

Now let $L$ be a distributive Artinian lattice of rank at most $n$. I want to produce a pair of $\perp^L$-independent generic extensions. For the sake of brevity and coherence I consider a group-theoretic situation only.

\begin{definition}
Let $G$ be a Polish group and $C\subset G$ a nonempty closed set. Say that $C$ has \emph{duplication index} $\geq n$ in $G$ if for all nonempty open sets $U\subset G$ and $O_i\subset C$ for $i\in n$ (the latter in the inherited topology on $C$), the set $\{\langle g+\bar c(i)\colon i\in n\rangle\colon g\in U$ and $\bar c\in\prod_iO_i\}\subset G^n$ has nonempty interior.
\end{definition}

\begin{example}
\label{aa}
Let $G=\mathbb{R}^n$ and $C=\mathbb{S}^{n-1}\subset\mathbb{R}^n$. Then $C$ has duplication index $\geq n$ in $G$.
\end{example}

\begin{proof}
It is enough to verify the statement for open neighborhoods $U$ of the origin. Shrinking $U$ and $O_i$ if necessary I may assume that $U$ is a ball around the origin, any tuple $\bar c\in\prod_iO_i$ is linearly independent, and in addition for any sequence $\bar d\in 4U^n$, the tuple $\bar c+\bar d$ is linearly independent. Then it follows that the map $f\colon \langle g, \bar c\rangle\mapsto \langle g+\bar c(i)\colon i\in n\rangle$ from $U\times \prod_iO_i\rangle\to (\mathbb{R}^n)^n$ is one-to-one.

To see this, choose a point $g\in 2U$ and an element $\bar c\in\prod_iU_i$ such that for each $i\in n$, $\bar c(i)+g\in O_i$ holds, and work to show that $g=0$. The assumption means that for every $i\in n$, $(\bar c(i)+g)\cdot (\bar c(i)+g)=1$. Distributing the dot product, this turns into $(\bar c_i-4g)\cdot g=0$. However, the vectors $\bar c_i-4g\in \mathbb{R}^n$ for $i\in n$ form a linearly independent tuple by the initial assumption on $U$ and $O_i$ for $i\in n$, so this system of equations (with $g$ viewed as an unknown vector) has only trivial solution $g=0$.

Thus, the map $f$ is one-to-one and definable in $\mathbb{R}$ with addition and multiplication. The rest of the argument follows easily from well-known results in real algebraic geometry \cite{vandendries:tame}. The domain of $f$ has dimension $n+(n-1)n=n^2$, so its range has the same (tame-topological) dimension. The range is then a definable subset of $(\mathbb R^n)^n$ of full dimension, which means that it contains a nonempty open subset.
\end{proof}

\begin{proposition}
\label{ab}
Let $G$ be a Polish group and $C\subset G$ be a nonempty closed set of duplication index $\geq n$. Let $L$ be a distributive Artinian lattice of rank $\leq n$. If $g\in G$ and $c\in C$ are elements mutually Cohen-generic over $x$, then $g\perp^L_xg+c$ holds.
\end{proposition}

\begin{proof}
The argument relies on a general claim of independent interest.

\begin{claim}
\label{duplicationclaim}
Suppose that $\bar y$ is an $n$-tuple $\forkindep$-independent over $x$, $\bar v$ is an $n$-tuple of elements of $L$ such that $\bar v(i)\in \vd_{x\bar y(i)}$, and $u\leq\bigwedge_i\bar v(i)$ is an indecomposable element. Then there is an $i\in n$ and an interpolant between $u$ and $\bar v(i)$ in $\vd_x$.
\end{claim}

\begin{proof}
For every set $b\subseteq n$, let $w_b=\bigwedge\{w\in L\cap \vd_{x\bar y\restriction b}\colon u\leq w\}$. By the Artinian property of $L$, $w_b$ exists as an element of $L$ and belongs to $\vd_{x\bar y\restriction b}$. It must be irreducible: otherwise, the decomposition of it into irreducible elements would belong to $\vd_{x\bar y\restriction b}$, and the irreducibility of $u$ would guarantee that one of the composants of $w_b$ must be above $u$. This would contradict the definition of $w_b$.

Fix $b\subseteq n$. The indepence of $\bar y$ together with Fact~\ref{fact1}(3) shows that the set $\{a\subseteq n\colon w_b\in\vd_{x\bar y\restriction a}\}$ is closed under intersection, so it contains an inclusion-smallest element $a_b$. There are now two cases.

\noindent\textbf{Case 1.} There is a nonempty set $b\subseteq n$ such that $a_b=0$. In such a case, pick $i\in b$ and observe that $w_b$ is an interpolant between $u$ and $\bar v(i)$ in $\vd_x$. This confirms the conclusion of the claim.

\noindent\textbf{Case 2.} Case 1 fails. In this case, by reverse recursion on $i\leq n$ build sets $b_i\subseteq n$ such that $b_n=n$ and $b_{i-1}=b_i$ with any single element of $a_{b_i}$ removed from $b_i$. It is then clear that $w_{b_i}\leq w_{b_{i-1}}$ and the equality does not occur, because $w_{b_i}$ does not belong to $w_{b_{i-1}}$ by the minimal choice of $a_{b_i}$. Thus, $\langle w_{b_i}\colon i\leq n\rangle$ is a strictly decreasing sequence of irreducible elements of $L$ of length $n+1$, contradicting the initial assumptions on $L$. The claim follows.
\end{proof}

Now, let $x$ be a set of ordinals such that $L\in\vd_x$. Let $g\in G$ and $c\in C$ be points mutually Cohen-generic over $x$ in $G$; I must argue that $g\perp^L_x g+c$. In order to do this, let $u_0\in \vd_{xg}$ and $u_1\in\vd_{xg+c}$ be elements of $L$ such that $u_0\leq u_1$; I must find an interpolant in $\vd_x$. In order to do this, let $F_0$, $F_1$ be $\vd$-functions such that $u_0=F_0(x, g)$ and $u_1=F_1(x, g+c)$.

Suppose towards a contradiction that the desired interpolant does not exist. Use the product property of $\forkindep$ to find $\vd_x$ sets $A_0, A_1$ such that $g\in A_0$, $c\in A_1$, and for every $g'\in A_0$ and $c'\in A_1$ such that $g'\forkindep_x c'$, $F_0(x, g')\leq F_1(x, g+c)$ holds and there is no interpolant between them in $\vd_x$.

Let $\bar d\in A_1^n$ be a tuple of elements of $A_1$ which is $\forkindep_{xg}$-independent. By the assumption on the duplication index of $C$, the tuple $\langle g+\bar d(i)\colon i\in n\rangle$ is $\forkindep_x$-independent. By the choice of the set $A_1$, for every $i\in n$, $u_0\leq F_1(x, g+\bar d(i))$ holds and there is no interpolant in $\vd_x$. This, however, immediately contradicts Claim~\ref{duplicationclaim}.
\end{proof}

\subsection{$\perp^L$ consistency results}

\begin{theorem}
\label{altheorem}
Let $d>0$ be a number and $L$ be a distributive Artinian lattice of rank $\leq d$. In every cofinally $\perp^L$-balanced forcing extension, if $A\subset\mathbb{R}^{d+1}$ is a co-meager set then there is $\eps>0$ such that for every positive real $\delta<\eps$ $A$ contains a pair of points of distance $\delta$.
\end{theorem}

\begin{proof}
Let $p\in P$ be a condition and $\tau$ be a $P$-name such that $p\Vdash\tau\subset\mathbb{R}^{d+1}$ is nonmeager. Since $P$ is cofinally balanced, there is a set of ordinals such that there is an open set $O\subset P$ which is balanced over $x$ and consists of conditions stronger than $p$. Consider the set $B=\{u\in\mathbb{R}^{d+1}\colon\exists q\in O\ q\Vdash\check u\in\tau\}$. This must be a non-meager subset of $\mathbb{R}^{d+1}$, since any condition in $O$ forces $\tau\subseteq\check A$. Since every set has the Baire property, there is a basic open ball $U\subset\mathbb{R}^{d+1}$ such that $A$ is co-meager in $U$. Let $\eps>0$ be any real number smaller than the radius of $U$. I will show that whenever $q\in O$ is a condition and $0<\gd<\eps$ is a real number, there are points $u_0, u_1\in U$ of Euclidean distance exactly $\gd$ and a condition $r\leq q$ such that $r\Vdash\check u_0, \check u_1\in\tau$. This will conclude the proof.

So, let $q\in O$ and $0<\gd<\eps$ be given. Let $y$ be a set of ordinals such that $q, \delta\in\vd_{xy}$, $\vd_{xy}$ contains a $G_\delta$-subset of $A$ which is dense in $U$, and there is a set $O'$ $\perp^L$-balanced over $xy$ and consisting only of conditions stronger than $q$. Consider the sphere $\mathbb{S}^{d}_\delta$ of radius $\delta$. Let $u\in \mathbb{R}^{d+1}$ and $v\in\mathbb{S}^{d}_\delta$ be points mutually Cohen-generic over ${xy}$ such that $u$ is witin $\eps-\delta$-distance from the center of $U$. Let $u_0=u$ and $u_1=u_0+v\in U$. Now, Example~\ref{aa} and Proposition~\ref{ab} shos that $u_0\perp^L_{xy}u_1$ holds. Note that the points $u_0, u_1\in U$ are Cohen-generic over ${xy}$ and they are both in the set $A$.

Now, let $O'_0=\{r\in O'\colon r\Vdash\check u_0\in\tau\}$ and $O'_1=\{r\in O'\colon r\Vdash\check u_1\in\tau\}$. These are both open subsets of $O'$, and they belong to $\vd_{xyu_0}$ and $\vd_{xyu_1}$ respectively. The set $O'_0$ is nonempty, because it is the intersection of $O'\in\vd_{xy}$ and $\{r\in O\colon r\Vdash\check u_0\in\tau\}\in\vd_{xu_0}$; these two sets have nonempty intersection by Proposition~\ref{extensionproposition1} and the fact that $y\forkindep[x]u_0$ (Fact~\ref{fact2}(2)). The set $O'_1$ is nonempty for the same reason with subscript $1$. It follows from Proposition~\ref{extensionproposition2} that $O'_0\cap O'_1\neq 0$, and any condition $r$ in this intersection forces both $u_0, u_1$ into $\tau$ as desired.
\end{proof}

\begin{corollary}
Let $d>0$ be a natural number. Relative to an inaccessible cardinal, it is consistent with ZF+DC that the rational distance graph in $\mathbb{R}^d$ has countable chromatic number while the same graph in $\mathbb{R}^{d+1}$ does not.
\end{corollary}

\begin{proof}
Let $L$ be the Artinian distributive lattice of algebraic subsets of $\mathbb{R}^d$ with intersection and union. It is well-known that the rank of $L$ is $d$; the coloring poset $P_d$ for the rational distance graph in $\mathbb{R}^d$ is $\perp^L$-balanced by Example~\ref{algd}. In the extension by $P_d$, for every partition of $\mathbb{R}^{d+1}$ into countably many pieces, one of them must be non-meager, therefore contains all sufficiently small distances by Theorem~\ref{altheorem}, so contains a rational distance.
\end{proof}

\section{Arity}

I close the paper with an example of an independence relation of higher arity. The independence relation of this section has served in \cite[Chapter 13]{z:geometric} to rule out discontinuous homomorphisms between Polish groups in certain balanced extensions of the Solovay model. It has many other applications.

\subsection{$\forkindep^{<d}$ definition}

\begin{definition}
Let $d>2$ be a number. The $d+1$-ary relation $\forkindep^{<d}$ is defined by $\forkindep_x^{<d}\bar y$ if $\bar y$ is a $d$-tuple of ordinals such that each subtuple of $\bar y$ of length smaller than $d$ is $\forkindep$-independent over $x$.
\end{definition}

\begin{proposition}
Let $d>2$ be a number. $\perp^{<d}$ is an independence relation with symmetry.
\end{proposition}

\begin{proof}
Symmetry follows immediately from the definition. Heredity and transitivity follow from the basic properties of independence of tuples as developed in \cite{z:reloadedA}.
\end{proof}

\begin{example}
Let $d\geq 2$ be a number. The poset $P_d$ of Example~\ref{algebraicexample} adding a coloring to the rational distance graph is cofinally $\perp^{<3}$-balanced.
\end{example}

\begin{proof}
By inspection of the proof of Example~\ref{algebraicexample}. The main point is that the rational distance graph has arity $2$. A more careful argument will show that any finite subset of $P_d$ in which any two elements have a common lower bound has a common lower bound, which implies that any balanced set is in fact $\forkindep^{<3}$-balanced.
\end{proof}

\subsection{$\forkindep^{<d}$ independent forcing extensions}

Given $d>2$, it is fairly easy to generate a wide variety of $\forkindep^{<d}$-independent $d$-tuples of generic extensions. This subsection contains only the most useful option in group-theoretic context.

\begin{example}
\label{groupexample}
Let $G$ be a group and $d>2$ be a number. If $\bar g\in G^d$ is a tuple, write $\prod\bar g$ for the product of entries of $\bar g$ in the order of increasing indexation. If $h\in G$, write $C^d_h=\{\bar g\in G^d\colon\prod\bar g=h\}$ ; this is a closed subset of $G^d$, therefore Polish in the inherited topology. If $x$ is a set of ordinals such that $G, h\in\vd_x$ and $\bar g\in C^d_h$ is a tuple Cohen-generic over $x$, then $\forkindep_x^{<d}\bar g$ holds.
\end{example}

\begin{proof}
For simplicity of notation, I will show that $\bar g\restriction d-1\in G^{d-1}$ is a tuple Cohen-generic over $x$; the case of other $d-1$-subtuples is analogous. 

Note that the function $\pi\colon G^{d-1}\to G$ defined by $\pi(\bar k)=(\prod\bar k)^{-1}h$ is continuous, and the set $C^d_h\subset G^d$ is exactly the graph of $\pi$. Thus, the projection from $C^d_h$ to $G^{d-1}$ is continuous and open; in consequence, the projection of Cohen-generic $d$-tuple in $C^d_h$ is a Cohen-generic $d-1$-tuple in $G^{d-1}$ as desired.
\end{proof}

\subsection{$\forkindep^{<d}$ consistency results}

The most prominent use of $\forkindep^{<d}$-balance is elimination of discontinuous homomorphisms between Polish groups from the generic extension through the following theorem.

\begin{theorem}
\label{grouptheorem}
Let $d>2$ be a number. $P$ be a cofinally $\perp^{<d}$-balanced poset. In the $P$-extension, whenever $G$ is a Polish group and $\bar D$ is a $d$-tuple of non-meager subsets of $G$, $\bar D!$ has nonempty interior.
\end{theorem}

\begin{proof}
Let $p\in P$ be a condition and $\tau_i$ for $i\in d$ be $P$-names such that $p\Vdash\tau_i\subseteq G$ is a non-meager subset. Use the balance assumption on $P$ to find a set $x$ of ordinals such that $p$ and $\tau_i$ for $i\in d$ belong to $\vd_x$ and there is a set $O\subset P$ which is balanced over $x$ and consists only of conditions stronger than $p$. For each $i\in d$ let $A_i=\{g\in G\colon\exists q\in O\colon q\Vdash\check g\in\tau_i$. Since any condition in the set $O$ forces $\tau_i\subseteq \check B_i$, it must be the case that $A_i$ is a nonmeager set. By the Baire axiom, there is a nonempty basic open set $B_i\subseteq G$ such that $A_i$ is co-meager in $B_i$. Let $V=\prod_iB_i$. $V\subset G$ is an open set, and it will be enough to show that every condition in $O$ forces $V\subseteq\prod_i\tau_i$.

To this end, let $q\in O$ be any condition and $g\in V$ be any element; I must produce a condition $r\leq q$ and elements $h_i\in G$ for $i\in d$ such that $\prod_ih_i=g$ and $r\Vdash\forall i\in d\ \check h_i\in\tau_i$. To this end, find a set $y$ of ordinals such that $q, g\in\vd_{xy}$, and such that there is an open set $O'\subset P$ which is $\perp^{<d}$-balanced over $xy$ and consists of conditions stronger than $q$. In $\vd_{xy}$, consider the set $C^d_g\subset G^d$ of Example~\ref{groupexample}, and find a sequence $\langle h_i\colon i\in d\rangle$ Cohen-generic over ${xy}$ in the set $\prod_iU_i$. For each $i\in d$, the point $h_i\in O_i$ is Cohen-generic over ${xy}$, and so over $x$. There is a string of consequences to this observation:

\begin{itemize}
\item by the choice of the open neighborhood $B_i$, $h_i\in A_i$;
\item by Fact~\ref{fact2}(2), $y\forkindep_xh_i$;
\item by the balance assumption on the set $O$, the set $O_i=\{r\in O'\colon r\Vdash\check h_i\in\tau_i\}\in\vd_{xyh_i}$ is nonempty.
\end{itemize}

\noindent Now, the sequence $\langle h_i\colon i\in d \rangle$ is $\forkindep^{<d}$-independent over $xy$ by Example~\ref{groupexample}. By the balance assumption on the set $O'$, the intersection $\bigcap_{i\in d}O_i$ is nonempty. Any condition $r$ in this intersection has the required properties.
\end{proof}

\begin{corollary}
\label{groupcorollary}
In cofinally $\perp^{<d}$-balanced extensions, all homomorphisms between Polish groups are continuous.
\end{corollary}

\begin{proof}
This is a simple descriptive set theoretic argument showing that the conclusion of Theorem~\ref{grouptheorem} implies the conclusion of Corollary~\ref{groupcorollary} abstractly in ZF.
Suppose $G, H$ are Polish groups, $\pi\colon G\to H$ is a homomorphism, and the conclusion of Theorem~\ref{grouptheorem} holds for nonmeager subsets of $G$. I will argue that $\pi$ is continuous.

First of all, observe that for every open neighborhood $V\subseteq H$ of the unit, the preimage $\pi^{-1}V$ is non-meager in $G$ as $G$ can be covered by countably many translates of it. Second, for every open neighborhood $V\subseteq H$ of the unit, there is a nonempty set $U\subseteq G$ such that $\pi''U\subseteq V$. To see this, let $V'\subseteq H$ be an open neighborhood of $1$ such that $(V')^d\subset V'$. The set $U'=\pi^{-1}V'\subseteq G$ is non-meager, so by the conclusion of Theorem~\ref{grouptheorem}, $(U')^d$ contains a nonempty open set $U\subseteq G$, and then $\pi''U\subseteq V$ must hold as $\pi$ is a homomorphism.

Finally, let me argue for the continuity of $\pi$. Let $V\subseteq H$ be an open set amd $g\in G$ is an element such that $\pi(g)=h\in V$ holds; I must find an open neighborhood $U$ fo $G$ such that $\pi''U\subseteq V$ holds. To this end, find an open neighborhood $V'\subseteq H$ of the unit such that $h+V'-V'\subset V$, use the conclusion of the previous paragraph to find a nonempty open neighborhood $U'\subseteq G$ such that $\pi''U'\subseteq V'$, and let $U=g+U'-U'$. This is an open neighborhood of $g$, and $\pi''U\subseteq V$ must hold as $\pi$ is a homomorphism. 
\end{proof}

\begin{corollary}
It is consistent with ZF+DC that for every $n\in\gw$ the space $\mathbb{R}^n$ can be decomposed into countably many pieces, neither of which contains a pair of points at a rational distance, and all homomorphisms between Polish groups are continuous.
\end{corollary}

\section{Appendix: axiomatization of the Solovay model}
\label{axiomatizationsection}

For reference reasons, this appendix records the axiomatization of the Solovay model as isolated in \cite{z:reloadedA}.

\begin{definition}
The \emph{geometric axiomatization of the Solovay model} consists of ZF+DC and the following axioms: ZF+DC, 

\begin{enumerate}
\item (definability) every set is definable from a set of ordinals;
\item (inaccessibility) there is no injective $\gw_1$-sequence of reals;
\item (Baire axiom) every set of reals has the Baire property;
\item (independence) there is a class ternary \emph{master independence relation} $\forkindep$ on sets of ordinals, definable from no parameters, with the following properties:

\begin{enumerate}
	\item (nontriviality) for every set $x$ of ordinals, $x\forkindep[x] x$ holds;
\item (symmetry) $y_0\forkindep[x]y_1$ implies $y_1\forkindep[x]y_0$;
\item (monotonicity) $y_0\forkindep[x]y_1$ and $y'_0\in\vd_{xy_0}$ implies $y'_0\forkindep[x]y_1$;
\item (transitivity) the conjunction $y_0\forkindep[x]y_1$ and $z\forkindep[xy_0]y_1$ implies $y_0z\forkindep[x]y_1$;
\item (extension) if $x$ is a set of ordinals and $A\in\vd_x$ is a nonempty set of sets of ordinals, thenthere are $y_0, y_1\in A$ such that $y_0\forkindep[x]y_1$;
\item (product) if $A\in\vd_x$ is a set, $y_0\forkindep[x]y_1$, and $\langle y_0, y_1\rangle\in A$, then there are $\vd_x$-sets $B_0, B_1$ containing $y_0, y_1$ respectively such that for all $y'_0\in B_0$ and $y'_1\in B_1$, $y'_0\forkindep[x]y'_1$ implies $\langle y'_0, y'_1\rangle\in A$.
\end{enumerate}
\end{enumerate}
\end{definition}

\noindent In this paper, the definability axiom is used at every turn. The inaccessibility axiom guarantees existence of points Cohen-generic over $x$ where $x$ is a set of ordinals. The Baire axiom is used to treat Cohen-generic points, even though in in \cite{z:reloadedA} it is applied much more broadly. Independence is used at every turn. \cite[Section 6]{z:reloadedA} also considers the Gandy--Harrington axiom, for which I have no use in the current paper. The main corollaries of the independence axioms are used at every turn:

\begin{fact}
\label{fact1}
Let $x$ be a set of ordinals.

\begin{enumerate}
\item If $y$ is a set of ordinals and $A\in\vd_x$ is a set of sets of ordinals, then there is $z\in A$ such that $y\forkindep[x]z$;
\item if $A$ is a $\vd_x$-set then $A$ is well-orderable if and only if $A\subset\vd_x$;
\item if $y_0\forkindep[x]y_1$ then $\vd_{xy_0}\cap\vd_{xy_1}=\vd_x$;
\item if $y_0\forkindep[x]y_1$ and $A_0\in\vd_{xy_0}$ and $A_1\in\vd_{xy_1}$ are disjoint subsets of $\vd_x$, then there is a set $B\in\vd_x$ such that $A_0\subseteq B$ and $A_1\cap B=0$.
\end{enumerate}
\end{fact}

\noindent In the present paper, Cohen-generic points of Polish spaces play a key role. The phrase ``$\vd_x$ contains the Polish space $Y$'' means that $\vd_x$ contains the set $Y$ together with some countable basis of its topology and its enumeration. The phrase ``$y$ is an element of $Y$ Cohen-generic over $x$'' means that $y$ belongs to no meager subset of $Y$ which belongs to $\vd_x$. The following fact ties these two notions together.

\begin{fact}
\label{fact2}
Let $x$ be a set of ordinals and $Y$ be a Polish space in $\vd_x$.

\begin{enumerate}
\item every meager $\vd_x$-subset of $Y$ is a subset of the union of all closed nowhere dense subsets of $Y$ in $\vd_x$;
\item if $z$ is a set of ordinals and $y\in Y$ is an element Cohen-generic over ${xz}$ then $y\forkindep_xz$.
\end{enumerate}
\end{fact}

\noindent The first item together with the inaccessibility axiom shows that the set of elements Cohen-generic over $x$ is a $\vd_x$ dense $G_\gd$ subset of $Y$. The second item shows that independence coincides with mutual genericity as far as Cohen forcing is concerned. If $Y_0$ and $Y_1$ are Polish spaces in $\vd_x$, the phrase ``$y_0\in Y_0$ and $y_1\in Y_1$ are mutually Cohen-generic over $x$'' means that the pair $\langle y_0, y_1\rangle\in Y_0\times Y_1$ is a Cohen-generic point over $x$. By Fact~\ref{fact2}(2), this is equivalent to $y_0$ and $y_1$ being separately Cohen-generic over $x$ and $y_0\forkindep_xy_1$.

Manipulations of Cohen-generic points in this paper can all be reduced to the following proposition.

\begin{proposition}
\label{mutualproposition}
Let $x$ be a set of ordinals, $Y_0, Y_1$ Polish spaces and $f\colon Y_0\to Y_1$ a continuous function with $Y_0, Y_1, f\in\vd_x$. Then

\begin{enumerate}
\item if $y_0\in Y_0$ and $y_1\in Y_1$ are points then $\langle y_0, y_1\rangle$ is a point of the graph of $f$ Cohen-generic over $x$ iff $y_0\in Y_0$ is Cohen generic over $x$ and $y_1=f(y_0)$;
\item if $f$-images of open sets have nonempty interior and $y_0\in Y_0$ is Cohen-generic over $x$ then $f(y_0)\in Y_1$ is Cohen-generic over $x$.
\end{enumerate}
\end{proposition}

\begin{proof}
For (2), observe that an $f$-preimage of a $\vd_x$ set nowhere dense in $Y_1$ is $\vd_x$ and nowhere dense in $Y_0$. For (1), note that the function $y\mapsto \langle y, f(y)\rangle$ is a $\vd_x$ homeomorphism between $Y_0$ and the graph of $f$ in the topology inherited from $Y_0\times Y_1$, and use (2) to this homeomorphism.
\end{proof}

\noindent For example, if $G$ is a Polish group in $\vd_x$ and $g_0g_1=g_2$ are three elements of it, then $g_0, g_1$ are mutually generic elements of $G$ iff $g_1$ and $g_2$ are, because by Proposition~\ref{mutualproposition}(1) the mutual genericity is in both cases equivalent to the triple $\langle g_0, g_1, g_2\rangle$ being a generic point of the set $\{\langle h_0, h_1, h_2\rangle\in G^3\colon h_0h_1=h_2\}$ equipped with the topology inherited from $G^3$.

\bibliographystyle{plain} 
\bibliography{odkazy,zapletal,shelah}

\end{document}